\documentclass{amsart}
\usepackage{latexsym,amsxtra,amscd,ifthen}
\usepackage{amsfonts}
\usepackage{color,graphicx}
\usepackage{verbatim}
\usepackage{amsmath}
\usepackage{amsthm}
\usepackage{amssymb}
\usepackage{url}

\theoremstyle{plain}

\newtheorem{theorem}{Theorem}
\newtheorem{lemma}[theorem]{Lemma}
\newtheorem{proposition}[theorem]{Proposition}
\newtheorem{corollary}[theorem]{Corollary}

\numberwithin{theorem}{section}
\numberwithin{equation}{theorem}

\theoremstyle{definition}
\newtheorem{definition}[theorem]{Definition}
\newtheorem{example}[theorem]{Example}
\newtheorem{remark}[theorem]{Remark}
\newtheorem{question}[theorem]{Question}
\newtheorem*{question*}{Question}

\newcommand{\fm}{\mathfrak{m}}

\DeclareMathOperator{\End}{End}

\DeclareMathOperator{\Hom}{Hom}
\DeclareMathOperator{\Aut}{Aut}
\DeclareMathOperator{\gr}{gr}

\DeclareMathOperator{\GKdim}{GKdim}

\DeclareMathOperator{\LND}{LND}

\begin{document}

\title[Zariski cancellation problem]
{Zariski cancellation problem for non-domain
noncommutative algebras}

\author{O. Lezama, Y.-H. Wang and J.J. Zhang}

\address{Lezama: Departamento de Matem{\'a}ticas
Universidad Nacional de Colombia, Sede Bogot{\'a}, Colombia}

\email{jolezamas@unal.edu.co}

\address{Wang: School of Mathematics,
Shanghai Key Laboratory of Financial Information Technology,
Shanghai University of Finance and Economics, Shanghai 200433, China}

\email{yhw@mail.shufe.edu.cn}

\address{Zhang: Department of Mathematics, Box 354350,
University of Washington, Seattle, Washington 98195, USA}

\email{zhang@math.washington.edu}

\begin{abstract}
We study Zariski cancellation problem for noncommutative
algebras that are not necessarily domains.
\end{abstract}





\maketitle


\section*{Introduction}
\label{yysec0}

This paper can be considered as a sequel to \cite{BZ1}, where
Bell-Zhang studied Zariski Cancellation Problem for noncommutative 
domains (in particular, for several families of Artin-Schelter 
regular algebras \cite{AS}). Many mathematicians have been making 
significant contributions to this research direction and related topics, 
see Brown-Yakimov \cite{BY}, Ceken-Palmieri-Wang-Zhang \cite{CPWZ1, CPWZ2},
Chan-Young-Zhang \cite{CYZ1, CYZ2}, Gaddis \cite{Ga}, Gaddis-Kirkman-Moore 
\cite{GKM}, Gaddis-Won-Yee \cite{GWY}, Levitt-Yakimov \cite{LY},
L{\"u}-Mao-Zhang \cite{LMZ}, Nguyen-Trampel-Yakimov \cite{NTY},
Tang \cite{Ta1, Ta2} and others.

Throughout let $k$ be a base commutative domain and everything is 
over $k$. If $k$ is a field, and then it is denoted by $\Bbbk$ 
instead. For example we assume that $k$ is a field $\Bbbk$ in 
Theorems \ref{yythm0.4} and \ref{yythm0.6}. Recall that an 
algebra $A$ is called {\it cancellative} 
if for any algebra $B$, an algebra isomorphism $A[t] \cong B[s]$ 
implies that $A\cong B$.  Here $t$ and $s$ are independent central 
variables. In the commutative case, the famous Zariski Cancellation 
Problem (abbreviated as ZCP) asks if the commutative polynomial ring 
$\Bbbk [x_1,\dots,x_n]$ is cancellative for all $n\geq 1$. It is 
well-known that $\Bbbk[x_1]$ is cancellative by a result of 
Abhyankar-Eakin-Heinzer \cite{AEH}, while $\Bbbk[x_1,x_2]$
is cancellative by Fujita \cite{Fu} and Miyanishi-Sugie \cite{MS} in
characteristic zero and by Russell \cite{Ru} in positive
characteristic. The original ZCP (for $n\geq 3$) was open for
many years. In 2014 Gupta \cite{Gu1, Gu2} settled the ZCP negatively
in positive characteristic for $n\geq 3$. The ZCP in characteristic
zero remains open for $n\geq 3$.

The ZCP is related to the Automorphism Problem,
Characterization Problem, Linearization Problem, Embedding
Problem, and Jacobian Conjecture, see a discussion in \cite{Kr, BZ1}.
In the noncommutative setting, the recent research suggests that
it is closely related to the Automorphism Problem of noncommutative
algebras, but it is unclear how to formulate the noncommutative
version of the Characterization Problem, Linearization Problem,
Embedding Problem and so on.

Some general methods were introduced in \cite{BZ1, LMZ} to attack
the ZCP for noncommutative domains. One effective method is to
use the discriminant to control the cancallation property \cite{BZ1}.
To extend the results in \cite{BZ1}, it is natural to ask if
the idea of the discriminant can be applied to non-domain
noncommutative algebras.

The first aim of this paper is to introduce several new concepts
and new methods to handle the ZCP for noncommutative algebras that are
not domain. We study {\it retractable} and {\it detectable}
properties of algebras and relate these properties to the
discriminant computation. Then we generalize some results
in \cite{BZ1} to the non-domain case. Results concerning 
the retractable and the detectable properties (and the rigidity 
introduced in \cite{BZ1}) can be summarized 
in the following diagram

$$\begin{CD}
{\text{$\LND$-rigid}} @>>> {\text{$\LND_Z$-rigid}} \\
@VVV @VVV\\
{\text{retractable}} @>>> {\text{$Z$-retractable}} \\
@. @VVV\\
@. {\text{detectable}}\\
@. @VVV\\
@. {\text{cancellative.}}
\end{CD}
$$

\bigskip

The second aim of this paper is to further understand which classes
of noncommutative algebra are cancellative. A proposition in
\cite[Proposition 1.3]{BZ1} states that, if the center of an algebra
$A$ is the base field $\Bbbk$, then $A$ is cancellative. This 
suggests

\begin{question}
\label{yyque0.1}
If the center of $A$ is finite dimensional over $\Bbbk$, is $A$ cancellative?
\end{question}

We partially answer the above question.

\begin{theorem}
\label{yythm0.2}
Suppose $A$ is strongly Hopfian {\rm{[}}Definition \ref{yydef3.4}{\rm{]}}
and the center of $A$ is artinian. Then $A$ is cancellative.
\end{theorem}

Question \ref{yyque0.1} is still open for non-Hopfian algebras.
Note that left (or right) noetherian algebras and locally finite
${\mathbb N}$-graded algebras are strongly Hopfian [Lemma \ref{yylem3.5}].
So Theorem \ref{yythm0.2} covers a large class of algebras. An immediate
consequence of this theorem is the following corollary.

\begin{corollary}
\label{yycor0.3}
If $A$ is left {\rm{(}}or right{\rm{)}} artinian, then $A$ is cancellative.
For example, every finite dimensional algebra over a base field $\Bbbk$ is
cancellative.
\end{corollary}

For non-artinian (and non-noetherain) algebras we have the following.

\begin{theorem}
\label{yythm0.4}
For every finite quiver $Q$, the path algebra $\Bbbk Q$ is
cancellative.
\end{theorem}

We refer to \cite{KL} for the definition and basic
properties of the Gelfand-Kirillov dimension (or
GKdimension for short). In this paper an affine algebra
means that it is finitely generated over the base ring
as an algebra. It is proved in \cite[Theorem 0.5]{BZ1} 
that, if $A$ is an affine domain of GKdimension
two over an algebraically closed field of characteristic
zero and if $A$ is not commutative, then $A$ is cancellative.
In contrast, noncommutative affine prime (non-domain) 
algebras of GKdimension two need not be cancallative 
[Example \ref{yyex1.3}(5)].
Further there are affine commutative domains of GKdimension
two that are not cancellative, see \cite{Da} or Example \ref{yyex1.3}(2).
By a result of Abhyankar-Eakin-Heinzer \cite[Theorem 3.3]{AEH}, every
affine commutative domain of GKdimension one is cancellative. These 
results suggest the following question.

\begin{question}
\label{yyque0.5}
Is every affine prime $\Bbbk$-algebra of GKdimension one cancellative?
\end{question}

We answer this question in the following case.

\begin{theorem}
\label{yythm0.6}
Let $\Bbbk$ be algebraically closed.
Then every affine prime $\Bbbk$-algebra of GKdimension one is
cancellative.
\end{theorem}

Our third aim is to introduce some basic questions. 
For example, Theorem \ref{yythm0.6} leads naturally 
to the following two questions.

\begin{question}
\label{yyque0.7}
Let $\Bbbk$ be algebraically closed of characteristic zero
and let $A$ be an
affine prime $\Bbbk$-algebra of GKdimension one. What can we say
about the (outer) automorphism group of $A$.
\end{question}

\begin{question}
\label{yyque0.8}
Let $A$ be an algebra of global dimension one
{\rm{(}}respectively, Krull dimension one{\rm{)}}. Is
then $A$ is cancellative?
\end{question}

A few other questions are listed in Section 5.

The paper is organized as follows. Section 1 contains definitions, known
examples and preliminaries. In Sections 2 and 3, we introduce retractable
and detectable properties. In Section 4, we prove the
Theorems \ref{yythm0.2}, \ref{yythm0.4} and \ref{yythm0.6}. In Section 5,
some comments and questions are given.

\section{Definitions and Preliminaries}
\label{yysec1}

We recall some definitions, known examples and basic properties.
First we copy the definition in \cite[Definition 1.1]{BZ1}.

\begin{definition}
\label{yydef1.1}
Let A be an algebra.
\begin{enumerate}
\item[(1)]
We call $A$ {\it cancellative} if any algebra isomorphism 
$\phi: A[t] \cong B[s]$ for an algebra $B$ implies that $A\cong B$.
\item[(2)]
We call $A$ {\it strongly cancellative} if, for each $n \geq 1$, 
any algebra isomorphism
$$A[t_1, \cdots, t_n]\cong B[s_1,\cdots,s_n]$$
for an algebra $B$ implies that $A \cong B$. 
\end{enumerate}
\end{definition}

In \cite[p. 311]{AEH}, the {\it strongly cancellative} property is called
the {\it invariant} property. In this paper we will study various versions
of the cancellative property.

\begin{definition}
\label{yydef1.2}
Let $A$ be an algebra and let $\LND(A)$ be the collection
of locally nilpotent $k$-derivations of $A$.
\begin{enumerate}
\item[(1)]
\cite{Mak}
The {\it Makar-Limanov invariant} of $A$ is defined to be
$$ML(A) = \bigcap_{\delta\in \LND(A)}\ker(\delta).$$
\item[(2)]
\cite{BZ1}
$A$ is called {\it $\LND$-rigid} if
$ML(A) = A$, or equivalently, $\LND(A) = \{0\}$.
\item[(3)]
\cite{BZ1}
$A$ is called {\it strongly $\LND$-rigid} if
$ML(A[t_1,\cdots, t_n]) = A$ for all $n\geq 1$.
\item[(4)]
The {\it Makar-Limanov center} of $A$ is defined to be
$$ML_Z(A) = ML(A)\cap Z(A)$$
where $Z(A)$ denotes the center of $A$.
\item[(5)]
We say that $A$ is $\LND_Z$-rigid if $ML_Z(A[t]) = Z(A)$, 
and {\it strongly $\LND_Z$-rigid} if $ML_Z(A[t_1,\cdots, t_n]) = Z(A)$ 
for all $n\geq 1$.
\end{enumerate}
\end{definition}

Note that we use $A[t]$ (instead of $A$) in the definition 
of $\LND_Z$-rigidity [Definition \ref{yydef1.2}(5)] which 
is slightly different from the $\LND$-rigidity in Definition 
\ref{yydef1.2}(2). Next we give some known examples. Recall 
that if an affine domain (containing $\mathbb{Q}$) of finite 
GKdimension is LND-rigid, then it is cancellative 
\cite[Theorem 0.4]{BZ1}.

\begin{example}
\label{yyex1.3}
Let $\Bbbk$ be an algebraically closed field of characteristic zero.
\begin{enumerate}
\item[(1)]
Let $A$ be an affine commutative domain of GKdimension
one. If $A$ is not isomorphic to $\Bbbk[x]$, then
$A$ is $\LND$-rigid \cite[Lemma 2.3]{CM}. In fact, 
it follows from \cite[Main Theorem, p.6]{CM} that
$A$ is strongly $\LND$-rigid. An earlier result of
\cite[Theorem 3.3 and Corollary 3.4]{AEH} states that
every affine commutative domain of GKdimension one is
cancellative.
\item[(2)]
Not every affine commutative domain of GKdimension two
is cancellative. The following example is due to 
Danielewski \cite{Da}. Let $n\geq 1$ and let $B_n$ be 
the coordinate ring of the surface $x^n y=z^2-1$ over 
$\Bbbk:={\mathbb C}$. Then $B_i\not\cong B_j$ if 
$i\neq j$, but $B_i[t]\cong B_j[s]$ for all $i,j\geq 1$, 
see \cite{Fi, Wi} for more details.
\item[(3)]
The famous Zariski Cancellation Problem (ZCP) asks if 
$\Bbbk[x_1,\dots,x_n]$ is cancellative. This is still 
open for any $n\geq 3$ when ${\rm{char}}\; \Bbbk=0$. 
Also see Example \ref{yyex1.6} when ${\rm{char}}\; \Bbbk>0$.
\item[(4)]
Here is an easy way of producing noncommutative algebras 
that are not cancellative. Starting with algebras $B_i$ 
in part (2). Let $A$ be a noncommutative algebra. If we 
can verify that $B_1\otimes A\not\cong B_2\otimes A$
(this is the case when the center of $A$ is ${\mathbb C}$), 
then $B_1\otimes A$ is not cancellative. For example, if 
$A$ is the $n$th Weyl algebra, $B_1\otimes A$ is not 
cancellative.
\item[(5)]
As a consequence of part (4), by taking $A$ to be the 
matrix algebra $M_{n}({\mathbb C})$ for some $n>1$, $M_n(B_1)$ 
is a noncommutative affine prime ${\mathbb C}$-algebra 
of GKdimension two that is not cancellative. See Theorem 
\ref{yythm0.6} for a related result. 
\end{enumerate}
\end{example}

One focus of this paper is to show that some classes of noncommutative
algebras are cancellative. The following two lemmas contain
elementary facts.

\begin{lemma}
\label{yylem1.4} Let $n$ be a positive integer. Let $A$ and $B$ be
two algebras such that $A[t_1,\cdots,t_n]\cong B[s_1,\cdots,s_n]$. 
Then the following hold.
\begin{enumerate}
\item[(1)]
$A$ is commutative if and only if $B$ is.
\item[(2)]
$A$ is left {\rm{(}}or right{\rm{)}} noetherian if and only if $B$ is.
\item[\rm (3)]
Assume that $A$ is left noetherian. Then,
$${\rm Kdim} A={\rm Kdim} B.$$
Here ${\rm Kdim}$ denotes the left Krull dimension.
\item[(4)]
$A$ is left {\rm{(}}or right{\rm{)}} artinian if and only if $B$ is.
\item[(5)]
$\GKdim A=\GKdim B$.
\item[(6)]
$A$ is a field if and only if $B$ is.
\item[(7)]
$A$ is a division ring if and only if $B$ is.
\item[(8)]
$A$ is a finite direct sum of fields if and only if $B$ is.
\item[(9)]
$A$ is a finite direct sum of division rings if and only if $B$ is.
\item[(10)]
$A$ is local left {\rm{(}}or right{\rm{)}} artinian if and only if $B$ is.
\end{enumerate}
\end{lemma}

%
%
%
%
%
%
%
%
%

For any algebra $A$, let $CI(A)$ denote the set of central
idempotents of $A$. There are two trivial central idempotents
$0,1\in CI(A)$. We say $A$ is {\it indecomposable} if $A$ is not
isomorphic to $A_1\oplus A_2$. It is clear that $A$ is indecomposable
if and only if $CI(A)=\{0,1\}$. The next lemma is clear.

\begin{lemma}
\label{yylem1.5}
Let $A$ be an algebra and let $Z(A)$ be the center of $A$.
\begin{enumerate}
\item[(1)]
$CI(A)=CI(Z(A))$.
\item[(2)]
The following are equivalent.
\begin{enumerate}
\item[(a)]
$CI(A)$ is finite.
\item[(b)]
$|CI(A)|= 2^n$ for some integer $n$.
\item[(c)]
$A=A_1\oplus A_2\oplus \cdots \oplus A_n$ where
each $A_i$ is an indecomposable algebra.
\end{enumerate}
\item[(3)]
If $A$ is ${\mathbb N}$-graded, then
$CI(A)=CI(A_0)$.
\item[(4)]
$CI(A[t_1,\cdots,t_n])=CI(A)$.
\end{enumerate}
\end{lemma}

%

To use the method of reduction modulo $p$, we need
consider the case when the base field has positive
characteristic. We start with the following example.

\begin{example}\cite{Gu1,Gu2}
\label{yyex1.6}
Let $\Bbbk$ be a base field of positive characteristic. Then
$\Bbbk[x_1,\cdots, x_n]$ is not cancellative for every $n\geq 3$.
This is a couterexample to the ZCP in positive characteristic.

Similar to Example \ref{yyex1.3}(4), one can construct noncommutative
algebras over a field of positive characteristic that are not cancellative.
\end{example}

When ${\rm char}\; \Bbbk>0$, the derivations in Definition \ref{yydef1.2}
need to be replaced by higher derivations, which we now recall.

\begin{definition}
\label{yydef1.7} Let $A$ be an algebra.
\begin{enumerate}
\item[(1)] \cite{HS}
A {\it higher derivation} (or {\it Hasse-Schmidt derivation})
on $A$ is a sequence of $k$-linear endomorphisms
$\partial:=\{\partial_i\}_{i=0}^{\infty}$ such that:
$$\partial_0 = id_A, \quad
\text{and} \quad  \partial_n(ab) =\sum_{i=0}^n
\partial_i(a)\partial_{n-i}(b)
$$
for all $a, b \in A$ and all $n\geq 0$. The collection of
higher derivations is denoted by ${\rm Der}^H(A)$.
\item[(2)]
A higher derivation is called {\it locally nilpotent} if
\begin{enumerate}
\item
given every $a\in A$ there exists $n\geq 1$ such that 
$\partial_i(a)=0$ for all $i\geq n$,
\item
the map
$$G_{\partial,t}:A[t]\to A[t] \qquad\qquad\qquad\qquad $$
defined by
$$ \qquad\qquad \qquad\qquad
a\mapsto \sum_{i=0}^{\infty} \partial_i(a)t^i, \;\; t\mapsto t,\;\;
{\text{for all $a\in A$,}}$$
is an algebra automorphism of $A[t]$.
\end{enumerate}
\item[(3)]
The collection of locally nilpotent higher derivations is denoted by
$\LND^H(A)$.
\item[(4)]
For every $\partial\in {\rm Der}^H(A)$, the kernel of $\partial$
is defined to be
$$\ker \partial =\bigcap_{i\geq 1} \ker \partial_i.$$
\end{enumerate}
\end{definition}

\begin{definition}
\label{yydef1.8}
Let $A$ be an algebra.
\begin{enumerate}
\item[(1)] \cite{Mak, BZ1}
The {\it Makar-Limanov$^H$ invariant} of $A$ is defined to be
\begin{equation}
\label{E1.8.1}\tag{E1.8.1}
ML^H(A) \ = \ \bigcap_{\delta\in {\rm LND}^H(A)} {\rm ker}(\delta).
\end{equation}
\item[(2)] \cite{BZ1}
We say that $A$ is \emph{$\LND^H$-rigid} if $ML^H(A)=A$,
or $\LND^H(A)=\{0\}$.
\item[(3)]\cite{BZ1}
$A$ is called \emph{strongly $\LND^H$-rigid} if $ML^H(A[t_1,\ldots,t_n])=A$, for all
$n\geq 0$.
\item[(4)]
The {\it Makar-Limanov$^H$ center} of $A$ is defined to be
\begin{equation}
\notag
ML^H_Z(A) =ML^H(A) \cap Z(A).
\end{equation}
\item[(5)]
$A$ is called \emph{strongly $\LND^H_Z$-rigid} if 
$ML^H_Z(A[t_1,\ldots,t_n])=Z(A)$, for all
$n\geq 0$.
\end{enumerate}
\end{definition}

In \cite[Theorem 3.3]{BZ1}, the strong $\LND$-rigidity 
($\LND^H$-rigidity) is the key to prove that several classes
of algebras are cancellative. However, if $A$ is not a domain 
(which is the case we are considering in the present paper), 
$A$ is hardly $\LND$-rigid (respectively, $\LND^H$-rigid).

\begin{lemma}
\label{yylem1.9}
Let $A$ be an algebra.
\begin{enumerate}
\item[(1)]
If $A$ has a non-central nilpotent element, then
$A$ is not $\LND$-rigid.
\item[(2)]
If $A$ is a prime algebra that is not a domain,
then every nilpotent element is not central.
As a consequence, $A$ is not $\LND$-rigid.
\item[(3)]
Let $A$ be prime. If $A$ is $\LND$-rigid, then
$A$ is a domain.
\end{enumerate}
\end{lemma}

\begin{proof}
(1) Let $x$ be a non-central nilpotent element.
Then $ad_x: a\mapsto xa-ax$ is a nonzero LND.
So $A$ is not $\LND$-rigid.

(2) Since $A$ is prime, $Z(A)$ is a domain.
So every nilpotent element is not in $Z(A)$.

Since $A$ is not a domain, there are $0\neq x,y\in A$
such that $xy=0$. Since $A$ is prime, $yAx\neq 0$.
Let $f=yax\neq 0$ for some $a\in A$. Then $f^2=0$.
By part (1) and the assertion we just proved,
$A$ is not $\LND$-rigid.

(3) This follows from part (2).
\end{proof}

By Lemma \ref{yylem1.9}, all prime algebras that are not domains are
not $\LND$-rigid. However, in the next section, we will show that many
non-domain prime algebras are $\LND_Z$-rigid.

\section{$\LND_Z$-rigidity controls retractability}
\label{yysec2}
In this and the next sections we study two properties
that are closely related to the cancellative property. Retractable
property, see the next definition, was studied in a paper
of Abhyankar-Eakin-Heinzer \cite{AEH}. In \cite[p. 311]{AEH}
it was called {\it strongly invariant} property.  But we changed
the term {\it invariant} to {\it strongly cancellative}, see
Definition \ref{yydef1.1}(2), and the term {\it strongly invariant}
to {\it strongly retractable}, see the next.

\begin{definition}
\label{yydef2.1}
Let $A$ be an algebra.
\begin{enumerate}
\item[(1)]
$A$ is called {\it retractable} if, for any algebra $B$, any
algebra isomorphism $\phi: A[t] \cong B[s]$ implies that
$\phi(A)=B$.
\item[(2)] \cite[p. 311]{AEH}
$A$ is called {\it strongly retractable} if, for any algebra $B$ and integer
$n\geq 1$, any algebra isomorphism $\phi: A[t_1,\dots, t_n] 
\cong B[s_1,\dots, s_n]$ implies that
$\phi(A)=B$.
\end{enumerate}
\end{definition}

\begin{example}
\label{yyex2.2}
Let $A$ be an affine commutative domain of GKdimension
one over a field $\Bbbk$.
By a result of \cite[Theorerm 3.3 and Corollary 3.4]{AEH},
if $A$ is not $\Bbbk'[x]$ for some field extension
$\Bbbk'\supseteq \Bbbk$, then it is strongly
retractabe; if $A=\Bbbk'[x]$ for some field extension
$\Bbbk'\supseteq \Bbbk$, then it is strongly
cancellative (but not retractable).
\end{example}

If $A$ is retractable, then every algebra automorphism
$\phi: A[t]\to A[t]$ restricts to an automorphism $\phi\mid_A: A\to A$.
This implies that there is a retraction of the
natural embedding
$$\Aut(A)\to \Aut(A[t]).$$
As a consequence, $\Aut(A[t])$ is a semidirect product of
$\Aut(A)\rtimes N$ for some normal subgroup $N\subseteq
\Aut(A[t])$.

Every strongly retractable algebra is retractable, and hence obviously,
cancellative. For every fixed integer $n\geq 1$, $A[t_1,\cdots,t_n]$
may be abbreviated as
$A[\underline{t}]$ and $B[s_1,\cdots,s_n]$ as $B[\underline{s}]$.

\begin{lemma}
\label{yylem2.3}
If $A$ is a finite direct sum of division rings, then $A$ is strongly
retractable.
\end{lemma}

\begin{proof}
Suppose that $\phi: A[\underline{t}]\to B[\underline{s}]$
is an algebra isomorphism. Similar to Lemma \ref{yylem1.4}(9), $B$ is a
direct sum of division algebras. Let $A^{\times}$ denote the set of invertible
elements in $A$. In this case, one can show that
$(A[\underline{t}])^{\times}=A^{\times}$ (the same is true for $B$). 
Let $A=D_1\oplus \cdots \oplus D_r$
for some $r\geq 1$, where each $D_i$ is a division ring. For every $a\in A$, write
\begin{center}
$a=(d_1,\dots,d_r)=(d_1,1,\dots,1)(1,0,\dots,0)+\cdots +(1,\dots,1,d_r)(0,\dots,0,1)$,
\end{center}
with $d_i\in D_i$, $1\leq i\leq r$. Observe that $(1,\dots1,d_i,1,\dots,1)$ is
either invertible or idempotent, and $(0,\dots,0,1,0,\dots,0)$ is idempotent,
hence  $A$ is generated by $(A[t])^{\times}$ and $CI(A[t])$. The same
properties hold for $B$. Since $\phi$ maps $(A[\underline{t}])^{\times}$ 
to $(B[\underline{s}])^{\times}$
and $CI(A[\underline{t}])$ to $CI(B[\underline{s}])$, we have $\phi(A)\subseteq B$.
By symmetry, $\phi(A)=B$.
\end{proof}

By \cite{BZ1} the Makar-Limanov invariants and $\LND$-rigidity control cancellation,
actually they control retractable property. The following result is
basically \cite[Theorem 3.3]{BZ1}.

\begin{lemma}
\label{yylem2.4}
Suppose $A$ is an affine domain of finite GK-dimension.
\begin{enumerate}
\item[(1)]
If ${\rm ML}(A[t])=A$ or ${\rm ML}^H(A[t])=A$, then $A$ is retractable.
\item[(2)]
If $A$ is strongly $\LND$-rigid or strong $\LND^H$-rigid, then $A$ is strongly
retractable.
\end{enumerate}
\end{lemma}

\begin{proof}
See the proof of \cite[Theorem 3.3]{BZ1}.
\end{proof}

The above lemma provides a lot of examples that are strongly
retractable.
We refer to papers \cite{BZ1, CPWZ1, CPWZ2, CYZ2, LY} for many
examples that are strongly $\LND$-rigid or strongly $\LND^H$-rigid.

For the rest of this section we want to deal with non-domain
case.

\begin{definition}
\label{yydef2.5}
Let $A$ be an algebra.
\begin{enumerate}
\item[(1)]
We call $A$ {\it $Z$-retractable}
if, for any algebra $B$, any algebra isomorphism $\phi: A[t]
\cong B[s]$ implies that $\phi(Z(A))=Z(B)$.
\item[(2)]
We call $A$ {\it strongly $Z$-retractable} if, for any algebra $B$
and integer $n\geq 1$, any algebra isomorphism
$\phi: A[t_1,\dots, t_n] \cong B[s_1,\dots, s_n]$
implies that $\phi(Z(A))=Z(B)$.
\end{enumerate}
\end{definition}

It is clear that (strongly) retractable algebras are
(strongly) $Z$-retractable. Copying the proof of
\cite[Theorem 3.3]{BZ1}, we also have the following,
which is similar to Lemma \ref{yylem2.4}.

\begin{lemma}
\label{yylem2.6} Let $A$ be an algebra such that the center
$Z(A)$ is an affine domain.
\begin{enumerate}
\item[(1)]
Suppose $ML_Z(A[t])=Z(A)$ or $ML^H_Z(A[t])=Z(A)$. Then
$A$ is $Z$-retractable.
\item[(2)]
Suppose that $A$ is strongly $\LND_Z$-rigid 
{\rm{(}}or strongly $\LND^H_Z$-rigid{\rm{)}}.
Then $A$ is strongly $Z$-retractable.
\end{enumerate}
\end{lemma}

By \cite[Section 5]{BZ1}, effectiveness of the 
discriminant controls $\LND^H$-rigidity. Since we 
will not compute explicitly discriminants
in this paper, to save space, we refer to \cite{CPWZ1, CPWZ2, BZ1}
for the definition of the discriminant. But we will use the idea of
``effectiveness'' in later sections, we now recall its definition next.
An algebra is called {\it PI} if it satisfies a polynomial identity.

\begin{definition} \cite[Definition 5.1]{BZ1}
\label{yydef2.7}
Let $A$ be a domain and suppose that $Y=\bigoplus_{i=1}^n kx_i$
generates $A$ as an algebra.
An element $f\in A$ is called \emph{effective} if the following
conditions hold.
\begin{enumerate}
\item[(1)]
There is an ${\mathbb N}$-filtration $\{F_i A\}_{\geq i}$ on $A$
such that the associated graded ring $\gr A$ is a domain
(one possible filtration is the trivial filtration $F_0 A=A$).
With this filtration we define the degree of elements in $A$, denoted
by $\deg_A$.
\item[(2)]
For every testing ${\mathbb N}$-filtered PI algebra $T$ with
$\gr T$ being an ${\mathbb N}$-graded domain and for every
testing subset $\{y_1,\ldots,y_{n}\}\subset T$ satisfying
\begin{enumerate}
\item[(a)]
it is linearly independent in the quotient $k$-module $T/k1_T$, and
\item[(b)]
$\deg y_i\geq \deg x_i$ for all $i$ and $\deg y_{i_0}>\deg x_{i_0}$
for some $i_0$,
\end{enumerate}
there is a presentation of $f$ of the form
$f(x_1,\ldots,x_{n})$ in the free algebra $k\langle x_1,\ldots,x_{n}
\rangle$, such that either $f(y_1,\ldots,y_n)$ is zero or
$\deg_{T} f(y_1,\ldots,y_n)>\deg_A f$.
\end{enumerate}
\end{definition}

The following is a special case of \cite[Lemma 5.3(6)]{BZ1}.

\begin{example} \cite[Lemma 5.3(6)]{BZ1}
\label{yyex2.8}
Let $A=\Bbbk[x]$. Every non-unit element $f\in A$ is effective in $A$.
To prove this we write $f=\sum_{i=0}^n a_i x^i$ where $a_i\in \Bbbk$ and
$a_n\neq 0$. Without loss of generality we might assume that $a_n=1$.
Consider $A$ as a filtered algebra generated by $Y=\Bbbk x$ by
defining $\deg x=1$. Clearly $\gr A\cong A$ is a domain. Let $T$ be
any testing filtered algebra and let $\{y\}$ be a testing  subset of
$T$. If $\deg y>1=\deg x$, then $\deg f(y)= n\deg y> n=\deg f$.
Therefore $f$ is effective.
\end{example}

More complicated examples of effective elements are given in
\cite[Section 5]{BZ1}.
There is another concept, called ``dominating'', see 
\cite[Definition 4.5]{BZ1} or \cite[Definition 2.1(2)]{CPWZ1}, 
that is similar to effective. Both of these properties control 
$\LND_Z^H$-rigidity. The following result is similar to 
\cite[Theorem 5.2]{BZ1}.

\begin{theorem}
\label{yythm2.9} Let $A$ be a PI prime algebra such that the discriminant
{\rm{(}}or $w$-discriminant for some $w${\rm{)}} over its center $Z$ is
effective {\rm{(}}respectively, dominating{\rm{)}} in $Z$. Suppose $Z$ is
an affine domain, then $A$ is strongly $\LND^H_Z$-rigid.
\end{theorem}

\begin{proof}
Since the proofs for the ``effective'' case and the ``dominating''
case are very similar, we only prove the ``effective'' case. We
also copy the proof of \cite[Theorem 5.2]{BZ1}.

Let $f$ be the discriminant of $A$ over $Z$, $d(A/C_A)$
(or $w$-discriminant $d_w(A/C_A)$).
Suppose $Z$ is generated by $\{x_1,\ldots,x_n\}$ as in
Definition \ref{yydef2.7} (as $f$ is effective).

Let $R=k[t_1,\ldots,t_d][t]$. By a prime algebra version of
\cite[Lemma 4.6(2)]{BZ1},
$$d_w(A\otimes  R/C_A\otimes R)=d_w(A/C_A)=f,$$
which is effective by hypothesis.
Let $\partial\in \LND^H(A[t_1,\cdots,t_d])$. By definition,
$G:=G_{\partial, t}\in \Aut_{k[t]}(A[t_1,\cdots,t_d][t])$.
For each $j$,
$$G(x_j)=x_j+\sum_{i\geq 1} t^i \partial_i(x_j).$$
Then $G$ is also an automorphism of $Z[t_1,\cdots,t_d][t]$.
We take the test algebra $T$ to be $Z[t_1,\cdots,t_d][t]$ where
the filtration on $T$ is induced by the filtration on $Z$ together
with $\deg t_s=1$ for all $s=1,\ldots,d$ and $\deg t=\alpha$
where $\alpha$ is larger than
$\deg \partial_i(x_j)$ for all $j=1,\ldots,n$ and all $i\geq 1$.
Now set $y_j=G(x_j)\in T$. By the choice of $\alpha$,
we have that
\begin{enumerate}
\item[(a)]
$\deg y_j\geq \deg x_j$, and that
\item[(b)]
$\deg y_j=\deg x_j$ if and only if $y_j=x_j$.
\end{enumerate}
If $G(x_j)\neq x_j$ for some $j$,
by effectiveness  as in Definition \ref{yydef2.7},
$\deg f(y_1,\cdots,y_n)>\deg f$. So $f(y_1,\cdots,y_n)
\neq_{A^\times} f$. But $f(y_1,\cdots,y_n)=G(f)=_{Z^\times} f$
by \cite[Lemma 4.4(4)]{BZ1}, a contradiction. Therefore
$G(x_j)=x_j$ for all $j$. As a consequence,
$\partial_i(x_j)=0$ for all $i\geq 1$, or equivalently, 
$x_j\in \ker \partial$.
Since $Z$ is generated by $x_j$'s, $Z\subset
\ker\partial$. Thus $Z\subseteq ML^H_Z(A[t_1,\cdots,t_d])$.
It is clear that $Z\supseteq ML^H_Z(A[t_1,\cdots,t_d])$,
see \cite[Example 2.4]{BZ1}. Therefore
$Z=ML^H_Z(A[t_1,\cdots,t_d])$ as required.
\end{proof}

\begin{corollary}
\label{yycor2.10}
If $A$ is a prime PI algebra over a field $\Bbbk$ such that
the center of $A$ is $\Bbbk[x]$. If the discriminant
of $A$ over $Z(A)$ is a nonunit in $Z(A)$, then $A$ is
$\LND^H_Z$-rigid.
\end{corollary}

\begin{proof} This follows from Example \ref{yyex2.8} and
Theorem \ref{yythm2.9}.
\end{proof}

\section{Detectability and cancellation}
\label{yysec3}

If $B$ is a subring of $C$ and $f_1,\dots,f_m$ are elements of $C$, then the
subring generated by $B$ and the set $\{f_1,\dots,f_m\}$ is denoted by $B\{f_1,\dots,f_m\}$.

\begin{definition}
\label{yydef3.1} Let $A$ be an algebra.
\begin{enumerate}
\item[(1)]
We call $A$ {\it detectable}
if, for any algebra $B$, an algebra isomorphism $\phi: A[t] \cong B[s]$ implies that
$B[s]=B\{\phi(t)\}$, or equivalently, $s\in B\{\phi(t)\}$.
\item[(2)]
We call $A$ {\it strongly detectable} if, for any algebra $B$ and any 
integer $n\geq 1$, an algebra isomorphism
$$\phi: A[t_1,\dots, t_n] \cong B[s_1,\dots, s_n]$$
implies that $B[s_1,\dots,
s_n]=B\{\phi(t_1),\dots,\phi(t_n)\}$, or equivalently, for each 
$i=1,\cdots, n$, $s_i\in B\{\phi(t_1),\dots,\phi(t_n)\}$.
\end{enumerate}
\end{definition}

In the above definition, we do not assume that $\phi(t)=s$.
Every strongly detectable algebra is detectable.
The polynomial ring $\Bbbk[x]$ is cancellative, but not
detectable.

\begin{lemma}
\label{yylem3.2} If $A$ is $Z$-retractable {\rm{(}}respectively, strongly
$Z$-rectractable{\rm{)}}, then it is detectable {\rm{(}}respectively, strongly
detectable{\rm{)}}.
\end{lemma}

\begin{proof} We only show the ``strongly'' version. The proof of the non-``strongly''
version is similar.

Suppose that $A$ is strongly $Z$-retractable. Let $B$ be any algebra 
such that $\phi: A[\underline{t}]\to B[\underline{s}]$ is an 
isomorphism. Since $A$ is strongly $Z$-retractable, $\phi$ restricts 
to an isomorphism $\phi\mid_{Z(A)}: Z(A)\to Z(B)$. Write
$f_i:=\phi(t_i)$ for $i=1,\dots,n$. Then
$$\begin{aligned}
Z(B)\{f_1,\cdots,f_n\}&=\phi(Z(A))\{\phi(t_1),\dots, \phi(t_n)\}\\
&=\phi(Z(A)\{t_1,\dots,t_n\})=\phi(Z(A)[\underline{t}])\\
&=\phi (Z(A[\underline{t}]))=Z(B[\underline{s}])\\
&= Z(B)[\underline{s}].
\end{aligned}
$$
Then, for every $i$, $s_i\in Z(B)[\underline{s}]= Z(B)\{f_1,\cdots,f_n\} \subseteq
B\{f_1,\dots,f_n\}$ as desired.
\end{proof}

Retractability is stronger than detectability. The next example shows
that these two properties are not equivalent.

\begin{example}
\label{yyex3.3} Let $A=\Bbbk[x,y]/(x^2=y^2=xy=0)$. By Theorem \ref{yythm4.1} in the
next section, $A$ is strongly detectable. But $A$ is not retractable as we show next.
Define an isomorphism $\phi: A[t]\to A[s]$ where $\phi(x)=x, \phi(t)=s$ and
$\phi(y)=y+sx$. Clearly $\phi(A)\neq A$, so $A$ is not retractable.
\end{example}

Next we show that detectability implies cancellation under some mild
conditions. We first recall a definition. 

\begin{definition}
\label{yydef3.4}
Let $A$ be a $k$-algebra.
\begin{enumerate}
\item[(1)]
We say $A$ is {\it Hopfian} if every $k$-algebra 
epimorphism from $A$ to itself
is an automorphism.
\item[(2)]
We say $A$ is {\it strongly Hopfian} if $A[t_1,\cdots, t_n]$ is Hopfian for every
$n\geq 0$.
\end{enumerate}
\end{definition}

In this definition, Hopfian property is dependent on the base
ring $k$. Hopfian algebras have been studied by 
several authors \cite{BRY, Le, Mal, OR}.
Every left noetherian algebra is strongly Hopfian, see the next lemma.
Some non-noetherian examples can be constructed using Lemma \ref{yylem3.5}(2,3).
An ${\mathbb N}$-graded $k$-algebra $A=\bigoplus_{i=0}^{\infty} A_i$ is
called {\it locally finite} if each homogeneous component $A_i$ is a
finitely generated $k$-module.

\begin{lemma}
\label{yylem3.5}
The following algebras are strongly Hopfian.
\begin{enumerate}
\item[(1)]
Left or right noetherian algebras.
\item[(2)]
Finitely generated locally finite ${\mathbb N}$-graded $k$-algebras
that are $k$-flat.
\item[(3)]
Prime affine $\Bbbk$-algebras satisfying a polynomial identity.
\end{enumerate}
\end{lemma}

\begin{proof} (1) Let $A$ be left noetherian,
then $A$ is Hopfian. Since $A[t_1,\cdots,t_n]$
is also left noetherian, $A$ is strongly Hopfian.

(2) 
If $A$ is an affine locally finite ${\mathbb N}$-graded algebra,
so is $A[t]$. Thus it suffices to show that $A$ is Hopfian.
Replacing $A$ by its localization $A\otimes_k \Bbbk$
where $\Bbbk$ is the fraction field of $k$, we might assume that
$k$ is a field $\Bbbk$. Note that the $k$-flatness insures that $A\otimes_k \Bbbk$
being $\Bbbk$-Hopfian implies that $A$ is $k$-Hopfian.

Next we assume that $k=\Bbbk$ which is a field.
In this case the assertion is a special case of 
\cite[Theorem 7, p.77]{BLK}, which is a consequence of a 
result of Mal'tsev (or Malcev) \cite{Mal}. We include a 
detailed proof here for the convenience of the reader.

We prove that $A$ is $\Bbbk$-Hopfian by contradiction.
Suppose that $\phi: A\to A$ is a surjective endomorphism of $A$ that has a
nonzero kernel. Fix $0\neq r\in \ker \phi$. Let $r\in \oplus_{i=0}^s A$,
for some integer $s\geq 0$,
and write $r=\sum_{i=0}^s r_i$ where some $r_i\in A_i$ are nonzero. Choose a
$\Bbbk$-linear basis ${\bf b}:=\{b_j\}_{j=1}^d$ of $\oplus_{i=0}^s A$
of homogeneous elements so that ${\bf b}$ contains all nonzero $r_i$s.

We use $R$ to denote subrings of $\Bbbk$ of the special form
${\mathbb Z}\{f_1,\cdots,f_w\}$ when ${\rm{char}}\; \Bbbk=0$
or of the special form ${\mathbb F}_p\{f_1,\cdots,f_w\}$ when ${\rm{char}}\; \Bbbk=p>0$.
Since $A$ is finitely generated, $A$ is generated by a finite set of homogeneous
generators, say ${\bf g}:=\{g_t\}_{t=1}^d$. We may assume that ${\bf g}$ is $\Bbbk$-linearly
independent and contains ${\bf b}$. Then there is a subring $R\subseteq \Bbbk$ of
the form specified as above such that $\phi$ restricts a
surjective $R$-algebra endomorphism of $A_R$, where $A_R$ denotes the $R$-subalgebra
$R\{g_1,\cdots,g_d\}$ of $A$ generated by ${\bf g}$.
Adding only finitely many new $f_i$ to $R$ if necessary we can assume that
every product of any two generators $g_{t_1},g_{t_2}\in {\bf g}$ has
coefficients in $R$ in terms of the basis ${\bf b}$ if such a product has
degree no more than $s$.
By the choice of ${\bf b}$ and ${\bf g}$, it is clear that $r\in A_R$
and that $\oplus_{i=0}^{s} (A_R)_i$ is a finitely generated free
$R$-module with $R$-basis ${\bf b}$. By the
construction of $R$, every simple factor ring $F$ of $R$ is a finite
field. The induced map $\phi_R\otimes F: A_R\otimes_R F\to A_R\otimes_R F$ is
still a surjective $F$-algebra endomorphism. Since $F$ is finite and $A_R\otimes_R F$
is locally finite over $F$, $A_R\otimes_R F$ is residually finite
in the sense that the ideals of finite index in ring have a trivial
intersection. By \cite[Theorem 3]{Le}, $A_R\otimes_R F$ is Hopfian (and then
$F$-Hopfian). This yields a contradiction because $\phi_R\otimes_R F$ is a
surjective endomorphism such that $\phi_R\otimes_R F(r\otimes_R 1)=0$,
but $r\otimes_R 1\neq 0$. Therefore the assertion follows.

(3) This is basically \cite[Corollary 2.3]{BRY}.
\end{proof}

We need Hopfian algebras in the next lemma.

\begin{lemma}
\label{yylem3.6} Suppose $A$ is strongly Hopfian.
\begin{enumerate}
\item[(1)]
If $A$ is detectable, then $A$ is cancellative.
\item[(2)]
If $A$ is strongly detectable, then $A$ is strongly cancellative.
\end{enumerate}
\end{lemma}

\begin{proof} We only prove (2).

Let $\phi: A[\underline{t}]\to B[\underline{s}]$ be an isomorphism.
Let $f_i=\phi(t_i)$ for $i=1,\cdots, n$. Then $f_i$ are central elements
in $B[\underline{s}]$. Thus $B\{f_1,\cdots,f_n\}$ is a homomorphic image
of $B[s_1,\cdots,s_n]$ by sending $s_i\mapsto f_i$. Suppose $A$ is strongly
detectable. Then $B\{f_1,\cdots, f_n\}=B[\underline{s}]$.
Then we have algebra homomorphisms
\begin{equation}
\label{E3.6.1}\tag{E3.6.1}
B[\underline{s}]\xrightarrow{\pi} B\{f_1,\cdots,f_n\}\xrightarrow{=} B[\underline{s}]
\cong A[\underline{t}].
\end{equation}
Since $A$ is strongly Hopfian, $A[\underline{t}]$ and then $B[\underline{s}]$
are Hopfian. Now \eqref{E3.6.1} implies that $\pi$ is an isomorphism.
As a consequence, $B\{f_1,\cdots, f_n\}=B[f_1,\cdots,f_n]$ considering $f_i$ as central
indeterminants in $B[f_1,\cdots,f_n]$. Now we have
$$A\cong A[\underline{t}]/(t_i)\xrightarrow{\overline{\phi}} 
B[f_1,\cdots, f_n]/(f_i)\cong B.$$
Therefore $A$ is strongly cancellative.
\end{proof}

Combining these results we have the following diagram for an 
algebra.

$$\begin{CD}
@. {\text{[Discriminant being] effective}}  \\
@. @V{\text{Remark \ref{yyrem3.7}(5)}}\not\Uparrow V{\text{Theorem \ref{yythm2.9}}}V\\
{\text{(strongly) $\LND$-rigid}} @>{\text{easy}}>
\not\Leftarrow{\text{Remark \ref{yyrem3.7}(4)}}> {\text{(strongly) $\LND_Z$-rigid}} \\
@V{\text{Lemma \ref{yylem2.4}}}V \Uparrow {\text{Remark \ref{yyrem3.7}(6)}}V 
@V \Uparrow {\text{Remark \ref{yyrem3.7}(7)}} V{\text{Lemma \ref{yylem2.6}}}V\\
{\text{(strongly) retractable}} @>{\text{easy}}
>\not\Leftarrow{\text{Remark \ref{yyrem3.7}(3)}}> {\text{(strongly) $Z$-retractable}} \\
@. @V{\text{Remark \ref{yyrem3.7}(2)}}\not\Uparrow V{\text{Lemma \ref{yylem3.2}}}V\\
@. {\text{(strongly) detectable}}\\
@. @V {\text{Remark \ref{yyrem3.7}(1)}}\not\Uparrow V {\text{Lemma \ref{yylem3.6}}}V\\
@. {\text{(strongly) cancellative}}
\end{CD}
$$
If ${\rm{char}}\; \Bbbk>0$, one should replace $\LND$ by $\LND^H$ in the above
diagram. 

\begin{remark}
\label{yyrem3.7} 
The following are easy.
\begin{enumerate}
\item[(1)]
$\Bbbk[x]$ is (strongly) cancellative, but not detectable.
\item[(2)]
The algebra in Example \ref{yyex3.3} is strongly detectable,
but not retractable (or $Z$-retractable).
\item[(3)]
Let $Z=\Bbbk[x^{\pm 1}]$ and $A=M_2(Z)$. Then the center of 
$A$ is $Z$. By Example \ref{yyex1.3}(1), $Z$ is strongly retractable.
Therefore $A$ is strongly $Z$-retractable. 
Consider the conjugation automorphism $\phi: A[t]\to A[t]$ determined by 
$$\phi: a \longrightarrow \begin{pmatrix} 1& t\\ 0&1\end{pmatrix} a
\begin{pmatrix} 1& -t\\ 0&1\end{pmatrix}, \qquad \forall a\in A.$$
It is easy to see that $\phi(A)\neq A$. Therefore $A$ is not retractable.
\item[(4)]
One can also check that the algebra $A$ in part (3) is 
strongly $\LND_Z$-rigid, but not $\LND$-rigid by Lemma \ref{yylem1.9}.

Here is another example. 
Let $A$ be an affine PI prime such that the discriminant
over its affine center is effective, for example, $A$ is in Example
\ref{yyex5.1}. By Theorem \ref{yythm2.9}, $A$ is strongly $\LND_Z$-rigid.
If $A$ is not a domain [Example \ref{yyex5.1}], then 
it is not $\LND$-rigid by Lemma \ref{yylem1.9}. 
\item[(5)]
Let $A=\Bbbk[x^{\pm 1}]$. Since $A$ is a commutative domain, 
the discriminant of $A$ over its center is trivial, whence,
not effective. But $A$ is strongly $\LND$-rigid and 
strongly $\LND_Z$-rigid by Example \ref{yyex1.3}(1).
\item[(6)]
The strong $\LND$-rigidity  (when  ${\rm{char}}\; \Bbbk=0$) and
the strong $\LND^H$-rigidity follow
from the strong retractability without any hypotheses as given in 
Lemma \ref{yylem2.4}. For example, 
if $A$ is not strongly $\LND$-rigidity and if ${\rm{char}}\; \Bbbk=0$,
then there is a locally nilpotent derivation of $\partial$ of
$A[t_1,\cdots,t_n]$, for some $n\geq 1$, such that $\ker \partial 
\not\supseteq A$. Consider the automorphism 
$$\exp(\partial, t): A[t_1,\cdots,t_n][t]\to A[t_1,\cdots,t_n][t]$$
determined by
$$\exp(\partial, t): a\mapsto \sum_{i=0}^{\infty} \frac{t^i}{i!} \partial^i(a), \quad 
t\mapsto t,$$
for all $a\in A[t_1,\cdots,t_n]$. Then $\exp(\partial, t)(A)\neq A$. 
Therefore $A$ is not strongly retractable. If $\partial$ is a 
higher derivation, one needs use the automorphism similar to the one in 
Definition \ref{yydef1.7}(2). 
\item[(7)]
The strong $\LND_Z$-rigidity (when ${\rm{char}}\; \Bbbk=0$) and 
the strong $\LND^H_Z$-rigidity follow
from the strong $Z$-retractability without any hypotheses as given in 
Lemma \ref{yylem2.6}. The argument in part (6) can be used with
some small modification.
\end{enumerate}
\end{remark}

One advantage of considering the detectable property is the following.

\begin{lemma}
\label{yylem3.8}
Let $A$ be an algebra with center $Z(A)$. If $Z(A)$ is {\rm{(}}strongly{\rm{)}}
detectable, so is $A$.
\end{lemma}

\begin{proof} Suppose $B$ is an algebra such that
$\phi: A[\underline{t}]\cong B[\underline{s}]$. Taking
the center, we have $\phi: Z(A)[\underline{t}]\cong Z(B)[\underline{s}]$.
Since $Z(A)$ is strongly detectable, $s_i\in Z(B)\{\phi(t_j)\}$ for all $i$.
Thus $s_i\in B\{\phi(t_j)\}$ for all $i$. This means that $A$ is strongly
detectable.
\end{proof}

Another property concerning detectable property is the following.

\begin{lemma}
\label{yylem3.9} Let $A$ be an algebra and $J$ be the prime radical
of $A$ that is nilpotent. If $A/J$ is {\rm{(}}strongly{\rm{)}}
detectable, so is $A$.
\end{lemma}

\begin{proof} Suppose $B$ is any algebra such that
$\phi: A[\underline{t}]\cong B[\underline{s}]$.
Since $J^n=0$ for some $n$, the prime radical of $A[\underline{t}]$ is
$J[\underline{t}]$ which is also nilpotent. This implies that the prime
radical of $B[\underline{s}]$ is nilpotent and of the form
$J(B)[\underline{s}]$ where $J(B)$ is the prime radical of $B$.
Modulo the prime radical, we obtain that
$$\phi': \quad (A/J)[\underline{t}]\cong (B/J(B))[\underline{s}]$$
where $\underline{t}$ and $\underline{s}$ still denote the images of
$\underline{t}$ and $\underline{s}$ in appropriate algebras respectively.
Since $A/J$ is strongly
detectable, $s_i \in (B/J(B))\{f_1,\cdots,f_n\}$ for all
$i$, which means that $s_i\in B\{f_1,\cdots,f_n\}$ modulo $J(B)[\underline{s}]$.
Here $f_i=\phi(t_i)$ for all $i$.
Another way of saying this is, for every $i$,
$$ f_i=s_i+\sum b_{d_1,\cdots,d_n} s_1^{d_1}\cdots s_n^{d_n} \in B\{f_1,\cdots,f_n\}$$
for  $b_{d_1,\cdots,d_n}\in J(B)$. The point is that $s_i=f_i$
modulo $J(B)[\underline{s}]$. Now we re-write $s_i$ as a
polynomial in $f_j$ with coefficients in $J(B)$, starting with,
$$\begin{aligned}
s_i& =f_i -\sum b_{d_1,\cdots,d_n} s_1^{d_1}\cdots s_n^{d_n} \\
&= f_i -\sum b_{d_1,\cdots,d_n} f_1^{d_1}\cdots f_n^{d_n} +
\sum b'_{d'_1,\cdots,d'_n} s_1^{d'_1}\cdots s_n^{d'_n}
\end{aligned}
$$
where $b'_{d'_1,\cdots,d'_n}$ are in $J(B)^2[\underline{s}]$.
This means that $s_i$ equals a polynomial of $f_1,\cdots,f_n$
modulo $J(B)^2[\underline{s}]$. By induction, $s_i$ equals a
polynomial of $f_1,\cdots,f_n$ modulo $J(B)^p[\underline{s}]$
for every $p\geq 1$. Since $J(B)$ is nilpotent,
$s_i$ is a polynomial of $f_1,\cdots,f_n$ (with coefficients
in $B$) when taking $p\gg 0$. Therefore $s_i\in B\{f_1,\cdots,f_n\}$ 
as required.
\end{proof}

We also mention an easy consequence of Lemma \ref{yylem3.8} and
\cite[Theorem 3.3]{AEH}.

\begin{proposition}
\label{yypro3.10}
If the center of $A$ is an affine domain of GKdimension one
that is not isomorphic to $\Bbbk'[x]$ for some field extension
$\Bbbk'\supseteq \Bbbk$, then $A$ is  strongly
detectable.
\end{proposition}

\begin{proof} By \cite[Theorem 3.3]{AEH}, $Z(A)$ is
strongly retractable. By Lemma \ref{yylem3.2},
$Z(A)$ is strongly detectable. The assertion
follows from Lemma \ref{yylem3.8}.
\end{proof}

\begin{lemma}
\label{yylem3.11}
Let $\Bbbk'$ is a field extension of $\Bbbk$.
If $A\otimes_{\Bbbk} \Bbbk'$ is {\rm{(}}strongly{\rm{)}} detectable
as an algebra over $\Bbbk'$, then $A$ is {\rm{(}}strongly{\rm{)}}
detectable as an algebra over $\Bbbk$.
\end{lemma}

\begin{proof} We only show the ``strongly'' version.
Suppose that $B$ is an algebra over $\Bbbk$ such that
$\phi: A[\underline{t}]\cong B[\underline{s}]$ is a 
$\Bbbk$-algebra isomorphism. Let $C$ be the $\Bbbk$-subalgebra of
$B\{\phi(t_i)\}_{i=1}^n$. Then $C\subseteq B[\underline{s}]$.
We claim that $C=B[\underline{s}]$. 

For any $\Bbbk$-module $M$, let $M_{\Bbbk'}$ denote $M\otimes_{\Bbbk}
\Bbbk'$. Consider the isomorphism
$\phi_{\Bbbk'}: A_{\Bbbk'}[\underline{t}]\cong
B_{\Bbbk'}[\underline{s}]$. By hypothesis,
$A_{\Bbbk'}$ is a {\rm{(}}strongly{\rm{)}} detectable
algebra over $\Bbbk'$. Then $B_{\Bbbk'}\{\phi(t_i)\}_{i=1}^n$
is $B_{\Bbbk'}[\underline{s}]=(B[\underline{s}])_{\Bbbk'}$.
Now one can easily show that $C_{\Bbbk'}=B_{\Bbbk'}\{\phi(t_i)\}_{i=1}^n$.
Thus $C_{\Bbbk'}=(B[\underline{s}])_{\Bbbk'}$, consequently,
$C=B[\underline{s}]$. Therefore $A$ is strongly
detectable as an algebra over $\Bbbk$.
\end{proof}

\section{Applications to the cancellation problem}
\label{yysec4}

In this section we will use the results proven in the previous
sections to show some classes of algebras are cancellative.
We first prove Corollary \ref{yycor0.3}.

\begin{theorem}
\label{yythm4.1} If $A$ is left {\rm{(}}or right{\rm{)}} artinian,
then $A$ is strongly detectable. As a consequence, $A$ is
strongly cancellative.
\end{theorem}

\begin{proof} Since $A$ is artinian, it is noetherian. By Lemma
\ref{yylem3.5}(1), it is strongly Hopfian. Now the strongly cancellative
property is a consequence of the strongly detectable property by
Lemma \ref{yylem3.6}(2). It remains to show that $A$ is strongly 
detectable. 

Let $J$ be the Jacobson radical of $A$. Then $J$ is also
the prime radical of $A$ and $J$ is nilpotent. By Lemma \ref{yylem3.9},
it suffices to show that $A':=A/J$ is strongly detectable.
Since $A'$ is a finite direct sum of matrix algebras over division rings,
$Z(A')$ is a finite direct sum of fields. By Lemmas
\ref{yylem2.3} and \ref{yylem3.2}, $Z(A')$ is strongly retractable,
and then strongly detectable. By Lemma
\ref{yylem3.8}, $A'$ is strongly detectable as required.
\end{proof}

\begin{theorem}
\label{yythm4.2}
Let $A$ be  an algebra with center $Z$. Suppose $J$ is the
prime radical of $Z$ such that
{\rm{(a)}} $J$ is nilpotent and {\rm{(b)}} $Z/J$ is a finite direct sum of
fields.
\begin{enumerate}
\item[(1)]
$A$ is strongly detectable.
\item[(2)]
If further $A$ is strongly Hopfian, then $A$ is strongly cancellative.
\end{enumerate}
\end{theorem}

Theorem \ref{yythm0.2} is an immediate consequence of the above theorem.

\begin{proof}[Proof of Theorem \ref{yythm4.2}]
(1) Similar to the proof of Theorem \ref{yythm4.1}, $Z$ is strongly detectable.
By Lemma \ref{yylem3.8} $A$ is strongly detectable.

(2) Follows from Lemma \ref{yylem3.6} and part (1).
\end{proof}

The following is an immediate consequence of the above theorem and Lemma
\ref{yylem3.5}(1).

\begin{corollary}
\label{yycor4.3}
Let $A$ be a left noetherian algebra such that its center
is artinian. Then $A$ is strongly detectable and strongly cancellative.
\end{corollary}

Next we will prove Theorem \ref{yythm0.4}. We start with the following lemma. We refer to
\cite{ASS} for basic definitions of quivers and their path algebra. Let $C_n$ be the cyclic quiver
with $n$ vertices and $n$ arrow connecting these vertices in one oriented direction. In
representation theory of finite dimensional algebras, quiver $C_n$ is also called {\it type
$\widetilde{A}_{n-1}$}. Let $0,1,\dots, n-1$ be the vertices of $C_n$, and $a_i: i\to i+1$ (in
${\mathbb Z}/(n)$) be the arrows in $C_n$. Then $w:= \sum_{i=0}^{n-1} a_{i}a_{i+1}\cdots a_{i+n}$
is a central element in $\Bbbk C_n$.

\begin{lemma}
\label{yylem4.4}
Let $Q$ be a finite quiver that is connected and let $\Bbbk Q$ be its path
algebra. Then
$$Z(\Bbbk Q)=\begin{cases}
\Bbbk & {\text{if $Q$ has no arrow,}}\\
\Bbbk[x]& {\text{if $Q=C_1$ or equivalently $\Bbbk Q=\Bbbk[x]$, }}\\
\Bbbk[w]& {\text{if $Q=C_n$ for $n\geq 2$,}}\\
\Bbbk &{\text{otherwise.}}
\end{cases}$$
\end{lemma}


\begin{proof} Since $\Bbbk Q$ is ${\mathbb N}$-graded, its center
$Z(\Bbbk Q)$ is also ${\mathbb N}$-graded.

If $Q$ has one vertex without arrow, $\Bbbk Q=\Bbbk$ and the center is $\Bbbk$.
If $Q=C_1$, then $\Bbbk Q=\Bbbk [x]$ and the center is $\Bbbk[x]$.
If $Q$ has one vertex with more than one arrows, then $\Bbbk Q$ is a free algebra
and in this case, its center is $\Bbbk$.

For the rest of the proof we suppose that $Q$ has at least two vertices.
Let $\{e_i\}_{i=1}^v$ be the vertices in $Q$ where $v\geq 2$. 

First of all, every central element of degree 0 is in the base field $\Bbbk$.
Let $f$ be a nonzero central element of degree $d>0$.
Write $f$ as $\sum c a_{s_1} a_{s_2}\cdots a_{s_d}$ where $c\in \Bbbk$
and $a_{s_i}$ are arrows in $Q$.
Since $f$ is central, $e_i f=fe_i$. Thus, 
$$f=(\sum_{i=1}^v e_i) f=\sum_{i=1}^v e_i f=\sum_{i=1}^v e_i f e_i=
\sum_{i=1}^{v} f_i$$ 
where $f_i=e_ife_i$. Since $f\neq 0$, $f_i\neq 0$ for some $i$. Without loss
of generality, we may assume that $f_1\neq 0$.

If $a$ is an arrow from $1$ to $i$ with $i\neq 1$, then
$$0\neq f_1 a=f a=af =af_i.$$
So $f_i\neq 0$. Similarly, if $a$ is an arrow from $i\neq 1$ to $1$,
then $f_i a=a f_1\neq 0$. Since $Q$ is connected, every vertex is
connected with the vertex 1. It follows by induction that 
$f_i\neq 0$. Further, for every arrow $b$ from $i$ to $j$, we have
\begin{equation}
\label{E4.4.1}\tag{E4.4.1}
f_i b= b f_j\neq 0.
\end{equation}
Write $f_i=\sum c a_{s_1}a_{s_2}\cdots a_{s_d}$ with some $c\in \Bbbk$.
Then equation \eqref{E4.4.1} implies that 
\begin{equation}
\label{E4.4.2}\tag{E4.4.2}
a_{s_1}=b
\end{equation}
for all possible arrow $b$ from $i$ to $j$ and all possible
nonzero terms $c a_{s_1}a_{s_2}\cdots a_{s_d}$ in $f_i$. 
Equation \eqref{E4.4.2} implies that
\begin{enumerate}
\item[(1)]
$b$ is unique, and 
\item[(2)]
$a_{s_1}$ is unique. 
\end{enumerate}
This means that, for every $i$, there is a unique $j$, denoted by 
$\sigma(i)$, such that there is a unique arrow from $i$ to 
$\sigma(i)$. Since $Q$ is connected with $v$ vertices,
$\sigma(i)\neq i$ for all $i$. This characterizes the quiver 
$C_n$. Now its routine to check that the center of $\Bbbk C_n$
is as described.

The above paragraph shows that, if $Q$ is not $C_n$, then there is no
nonzero central element of positive degree. Therefore, $Z(\Bbbk Q)=\Bbbk$.
This finishes the proof.
\end{proof}

\begin{lemma}
\label{yylem4.5}
Let $Q=C_n$ for $n\geq 2$.
\begin{enumerate}
\item[(1)]
$\Bbbk Q$ is a prime  algebra of GKdimension 1 that is not Azumaya.
\item[(2)]
$\Bbbk Q$ is strongly detectable and strongly cancellative.
\end{enumerate}
\end{lemma}

\begin{proof}
(1) By a direct computation.

(2) If $\Bbbk$ is algebraically closed, this is a special case of
Theorem \ref{yythm4.12}(2), which we will prove later.
If $\Bbbk$ is not algebraically closed, let $\Bbbk'$ be the closure of
$\Bbbk$. By Theorem \ref{yythm4.12}(2), $\Bbbk' Q$ is strongly detectable
over $\Bbbk'$. By Lemma \ref{yylem3.11}, $\Bbbk Q$ is strongly detectable
over $\Bbbk$, and then strongly cancellative by Lemmas \ref{yylem3.5}(2)
and \ref{yylem3.6}(2).
\end{proof}

We need an elementary lemma. Two idempotents $e_1$ and $e_2$ in an 
algebra $A$ are {\it counital} if $e_1+e_2=1$.

\begin{lemma}
\label{yylem4.6}
Let $A$ and $B$ be two algebras.
\begin{enumerate}
\item[(1)]
If $A$ and $B$ are {\rm{(}}strongly{\rm{)}} cancellative, so is $A\oplus B$.
\item[(2)]
If $A$ and $B$ are {\rm{(}}strongly{\rm{)}} retractable, so is $A\oplus B$.
\item[(3)]
If $A$ and $B$ are {\rm{(}}strongly{\rm{)}} detectable, so is $A\oplus B$.
\end{enumerate}
\end{lemma}

\begin{proof} We only prove part (1). The proof of parts (2,3)
is similar.

Suppose $\phi: (A\oplus B)[\underline{t}]\cong C[\underline{s}]$ is an
algebra isomorphism, and let $e_1$ and
$e_2$ be two counital idempotents corresponding to
the decomposition $A\oplus B$. Let $f_i=\phi(e_i)$
for $i=1,2$. By Lemma \ref{yylem1.5}(4), $f_1$ and $f_2$ are
two counital idempotents of $C$. Thus
$C=C_1\oplus C_2$ and $A[\underline{t}]\cong C_1[\underline{s}]$ and
$B[\underline{t}]\cong C_2[\underline{s}]$. Since $A$ and $B$ are
strongly cancellative,
$A\cong C_1$ and $B\cong C_2$, which implies that
$A\oplus B\cong C_1\oplus C_2=C$. The proof of non-``strongly''
version is similar.
\end{proof}

The next is Theorem \ref{yythm0.4}.

\begin{theorem}
\label{yythm4.7} Let $Q$ be a finite quiver and let $A$ be the path algebra
$\Bbbk Q$. Then $A$ is strongly cancellative. If further $Q$ has no
connected component being $C_1$, then $A$ is strongly detectable.
\end{theorem}

\begin{proof} By Lemma \ref{yylem4.6}, we may assume that $Q$ is
connected.

If $Q=C_1$, then $A=\Bbbk[x]$ and the assertion follows
from \cite[Theorem 3.3 and Corollary 3.4]{AEH}.

If $Q=C_n$, then this is Lemma \ref{yylem4.5}(2).

If $Q\neq C_n$ for any $n\geq 1$, then by Lemma \ref{yylem4.4},
the center of $A$ is $\Bbbk$. By Theorem \ref{yythm4.2}(1),
$A$ is strongly detectable. Since $A$ is ${\mathbb N}$-graded
and locally finite, it is strongly Hopfian by Lemma \ref{yylem3.5}(2).
By Theorem \ref{yythm4.2}(2), $A$ is strongly cancellative. This
completes the proof.
\end{proof}

For the rest of this section we consider affine prime
algebras of GKdimension one and prove Theorem \ref{yythm0.6}.
For simplicity, we assume that $\Bbbk$ is algebraically closed.
It is not immediately clear to us if this assumption can be
removed. However, this assumption definitely makes some of 
argument and even language easier.

Let $A$ be an affine prime algebra of GKdimension one.
By a result of Small-Warfield \cite{SW}, $A$ is a finitely generated
module over its affine center. As a consequence,
$A$ is noetherian.

Let $R$ be a commutative
algebra, an $R$-algebra $A$ is called {\it Azumaya} if $A$ is a
finitely generated faithful projective $R$-module and the
canonical morphism
\begin{equation}
\label{E4.7.1}\tag{E4.7.1}
 A\otimes_R A^{op}\to \End_R(A)
\end{equation}
is an isomorphism. Note that the Brauer group of a commutative
algebra $R$, denoted by $Br(R)$, is the set of
Morita-type-equivalence classes of Azumaya algebras over $R$,
in other words, $Br(R)$ classifies Azumaya algebras over $R$
up to an equivalence relation \cite{AG}. See \cite{Sc} for 
some discussion about the Brauer group. Another way of defining 
an Azumaya algebra is the following (when $\Bbbk$ is algebraically 
closed). We refer to \cite[Introduction]{BY} for a discussion.

\begin{definition} \cite[Introduction]{BY}
\label{yydef4.8}
Suppose that $\Bbbk$ is algebraically closed. Let $A$ be an 
affine prime $\Bbbk$-algebra which is a finitely generated 
module over its affine center $Z(A)$. Let $n$ be the 
PI-degree of $A$, which is also the maximal possible 
$\Bbbk$-dimension of irreducible $A$-modules.
\begin{enumerate}
\item[(1)]
The {\it Azumaya locus} of $A$, denoted by ${\mathcal A}(A)$,
is the dense open subset of ${\rm{Maxspec}}\; Z(A)$ which
parametrizes the irreducible $A$-modules of maximal $\Bbbk$-dimension.
In other word, $\fm \in {\mathcal A}(A)$ if and only if
$\fm A$ is the annihilator in $A$ of an irreducible $A$-module
$V$ with $\dim V = n$, if and only if $A/\fm A\cong M_{n}(\Bbbk)$.
\item[(2)]
If ${\mathcal A}(A)={\rm{Maxspec}}\; Z(A)$, $A$ is called
{\it Azumaya}.
\end{enumerate} 
\end{definition}

There are several other equivalent definitions of an Azumaya algebra,
but we will only use Definition \ref{yydef4.8}(2).

For affine prime algebras of GKdimension one, we have the following.

\begin{lemma}
\label{yylem4.9}
Let $A$ be an affine prime algebra of GKdimension one
over an algebraically closed field $\Bbbk$. Then one of the following
holds.
\begin{enumerate}
\item[(1)]
$Z(A)\not\cong \Bbbk[x]$.
\item[(2)]
$Z(A)\cong \Bbbk[x]$ and $A$ is not Azumaya.
\item[(3)]
$Z(A)\cong \Bbbk[x]$ and $A$ is Azumaya. In this case,
$A$ is a matrix algebra over $\Bbbk[x]$.
\end{enumerate}
\end{lemma}

\begin{proof} Only non-trivial assertion is the second statement
in part (3).

Now we consider $R=\Bbbk[x]$.
By \cite[Theorem 7.5]{AG}, $\Bbbk[x]$ has trivial
Brauer group when $\Bbbk$ is algebraically closed.
Since every projective $\Bbbk[x]$-module is free,
$\Hom_{\Bbbk[x]}(E,E)$ is a matrix algebra over
$\Bbbk[x]$ for every finitely generated projective
$\Bbbk[x]$-module $E$.
By definition \cite[Section 5]{AG}, every Azumaya
algebra $A$ over $\Bbbk[x]$ is  Morita equivalent
to $\Bbbk[x]$. Again, by the fact every finitely 
generated projective $\Bbbk[x]$-module is free, 
$A$ is a matrix algebra over $\Bbbk[x]$.
\end{proof}

In fact, if ${\rm char}\; \Bbbk=0$, by \cite[Proposition 7.7]{AG},
every Azumaya algebra over $\Bbbk [x_1,\cdots,x_n]$
is a matrix algebra over $\Bbbk [x_1,\cdots, x_n]$.
But this is not true when ${\rm char}\; \Bbbk>0$ and 
$n\geq 2$, see Remark \ref{yyrem5.2}.

\begin{proposition}
\label{yypro4.10}
Suppose that $Z(A)=\Bbbk[x]$ and that $A$ is not Azumaya.
\begin{enumerate}
\item[(1)]
If ${\rm char}\; \Bbbk=0$,
then $A$ is strongly $\LND^H_Z$-rigid.
\item[(2)]
In general, $A$ is $Z$-retractable.
\end{enumerate}
As a consequence, $A$ is strongly detectable
and strongly cancellative.
\end{proposition}

\begin{proof}
The consequence follows from Lemmas \ref{yylem2.6},
\ref{yylem3.2} and \ref{yylem3.6}.

(1) Thanks to a wonderful result of Brown-Yakimov
\cite[Main Theorem]{BY}, the discriminant of $A$ over $Z(A)$
is nonunit (this happens if and only if $A$ is not Azumaya).
Since $Z(A)=\Bbbk[x]$, by Corollary \ref{yycor2.10},
$A$ is strongly $\LND^H_Z$-rigid.

(2) If ${\rm char}\; \Bbbk$ is positive, a trace on $A$ might 
not be representation theoretic in the sense of 
\cite[Section 2.1]{BY}. In this case, we can not use 
\cite[Main Theorem]{BY}. So we need a slightly different approach. 
Let $C$ be a PI prime algebra that is finitely generated over 
its center $Z(C)$. Let ${\mathcal N}(C)$ denote the complement 
of the Azumaya locus of $C$ and let $N(C)$ be the ideal of $Z(C)$
that corresponds to ${\mathcal N}(C)$. Then $N(C)$ is a
nonzero proper ideal of $Z(C)$ if $C$ is not Azumaya.

Let $D$ be the polynomial extension
$$D:=C[t_1,\cdots,t_n]\cong C\otimes \Bbbk[t_1,\cdots,t_n]$$
for some $n\geq 1$. Then $D$ is a finitely generated module 
over its center $Z(D)=Z(C)[t_1,\cdots,t_n]$. Using Definition
\ref{yydef4.8}, one can routinely show that
\begin{equation}
\label{E4.10.1}\tag{E4.10.1}
N(C[t_1,\cdots,t_n])=N(C)[t_1,\cdots,t_n].
\end{equation}

Now let $A$ be the algebra in the proposition
and let $B$ be an algebra such that $\phi: A[t_1,\cdots,t_n]
\cong B[s_1,\cdots, s_n]$. We claim that $\phi(Z(A))=Z(B)$.
First of all, $\phi$ induces an isomorphism
$Z(A)[t_1,\cdots,t_n]\cong Z(B)[s_1,\cdots,s_n]$.
Since $Z(A)=\Bbbk[x]$ is strongly cancellative,
$Z(B)=\Bbbk[y]$.
Since $A$ is not Azumaya, there is a $f\in \Bbbk[x]\setminus \Bbbk$
such that $N(A)=f \Bbbk[x]$. Similarly, $N(B)=g
\Bbbk[y]$ for some $g\in \Bbbk[y]\setminus \Bbbk$. It follows from 
\eqref{E4.10.1} that 
$$\begin{aligned}
g \Bbbk[y,s_1,\cdots,s_n]
&= N(B)[s_1,\cdots,s_n]\\
& =N(B[s_1,\cdots,s_n])\\
&\cong N(A[t_1,\cdots,t_n]) \qquad\qquad ({\rm{via}} \;\; \phi^{-1})\\
&=f \Bbbk[x,t_1,\cdots,t_n].
\end{aligned}
$$
Thus $\phi$ maps $f$ to a scalar multiple of $g$.
Define
$$W_f(A[t_1,\cdots, t_n])=\{ w \in Z(A[t_1,\cdots,t_n])\mid
\exists v \in Z(A[t_1,\cdots,t_n]), s.t. \; wv=f\}.$$
Recall that $Z(A[t_1,\cdots,t_n])=\Bbbk[x,t_1,\cdots,t_n]$
and $f\in Z(A)=\Bbbk[x]$ is not a scalar.
Then $W_f(A[t_1,\cdots,t_n])$ is a subset of
$\Bbbk[x]$ containing at least an element of 
form $a+x$ for some $a\in \Bbbk$. Similarly, 
$W_g(B[t_1,\cdots,t_n])$ is a subset of
$\Bbbk[y]$ containing an element of 
form $a'+y$ for some $a'\in \Bbbk$. Since 
$\phi$ maps $W_f(A[t_1,\cdots, t_n])$ to 
$W_g(B[t_1,\cdots,t_n])$ and since $Z(A)$ (respectively, $Z(B)$)
is generated by $a+x$ (respectively, $a'+y$), we have
$\phi(Z(A))=Z(B)$ as required.
\end{proof}

\begin{lemma}
\label{yylem4.11}
The matrix algebra $A:=M_n(\Bbbk [x])$ is strongly cancellative.
\end{lemma}

\begin{proof}
If $n=1$, this is \cite[Theorerm 3.3 and Corollary 3.4]{AEH}.
Now assume that $n>1$.
Suppose $B$ is an algebra such that $\phi: A[\underline{t}]
\cong B[\underline{s}]$ is an isomorphism. Then $Z(A)[\underline{t}]
\cong Z(B)[\underline{s}]$. Since $Z(A)=\Bbbk[x]$, by
\cite[Theorerm 3.3 and Corollary 3.4]{AEH}, $Z(B)\cong \Bbbk[x]$.
It is clear that $B$ is prime of GKdimension one. If $B$ is
not Azumaya over $Z(B)$, then $B$ is strongly cancellative by
Proposition \ref{yypro4.10}. As a consequence, $A\cong B$. If 
$B$ is Azumaya over $Z(B)$, then $B$ is a matrix algebra over 
$\Bbbk[x]$ by Lemma \ref{yylem4.9}(3), say $B=M_{n'}(\Bbbk[x])$. 
Now algebra isomorphisms 
$$M_n(\Bbbk[x,t_1,\cdots,t_n])\cong A[\underline{t}]\cong
B[\underline{s}]\cong M_{n'}(\Bbbk[x,s_1,\cdots,s_n])$$
implies that $n=n'$. Therefore $A\cong B$ as desired.
\end{proof}

Next we are ready to prove Theorem \ref{yythm0.6}.

\begin{theorem}
\label{yythm4.12}
Let $A$ be an affine prime algebra of GKdimension
one over an algebraically closed field.
\begin{enumerate}
\item[(1)]
$A$ is strongly cancellative.
\item[(2)]
If either $Z(A)\neq \Bbbk[x]$ or
$A$ is not Azumaya, then $A$ is
strongly detectable.
\end{enumerate}
\end{theorem}

\begin{proof} By Lemma \ref{yylem4.9}, there are three cases
to consider.

Case 1: $Z(A)\not\cong \Bbbk[x]$. By
\cite[Theorerm 3.3 and Corollary 3.4]{AEH}, $Z(A)$ is
strongly retractable. By Lemma \ref{yylem3.2}, $A$ is
strongly detectable, and by Lemma \ref{yylem3.6}(2),
$A$ is strongly cancellative.

Case 2: $Z(A)\cong \Bbbk[x]$ and $A$ is not Azumaya.
The assertion follows from Proposition \ref{yypro4.10}.

Case 3: $Z(A)\cong \Bbbk[x]$ and $A$ is Azumaya. The assertion
follows from Lemma \ref{yylem4.9}(3) and \ref{yylem4.11}.
\end{proof}

\section{Examples, remarks and questions}
\label{yysec5}

First of all, it is quite easy to produce a large family 
of prime, but non-domain, PI rings that are strongly cancellative. 

\begin{example}
\label{yyex5.1}
Let $R$ be an affine commutative domain and let $f$ be a product
of a set of generating elements of $R$. Let
$$A=\begin{pmatrix} R & fR \\R &R\end{pmatrix}.$$
Then the discriminant of $A$ over its center $R$ is $-f^2$.
It easy to see that $-f^2$ is a effective element in $R$. So
$A$ is strongly $\LND^H_Z$-rigid [Theorem \ref{yythm2.9}].
As a consequence, $A$ is strongly $Z$-retractable, detectable
and cancellative. 
\end{example}

Recall that the Brauer group of a commutative domain $A$ is 
denoted by $Br(A)$. 

\begin{remark}
\label{yyrem5.2} 
Let $\Bbbk$ be an algebraically closed field of characteristic 
zero and let $A$ be a commutative affine regular domain over 
$\Bbbk$. By \cite[Proposition 7.7]{AG}, $Br(A)$ is naturally 
isomorphic to $Br(A[t])$. Consequently, $Br(\Bbbk[x_1,\cdots,x_n])$
is trivial for all $n\geq 0$. If $A[t]\cong B[s]$ for some 
algebra $B$, then $B$ is also a commutative affine regular domain. 
Therefore
$$Br(A)\cong Br(A[t])\cong Br(B[s])\cong Br(B).$$
We are wondering if there are other connections between the 
cancallative property of $A$ and the structure of $Br(A)$. For 
example, does what kind of structure of $Br(A)$ force that $A$ 
cancallative?

If $\Bbbk$ is an algebraically closed field of positive 
characteristic, the situation is different. First of all, 
$Br(\Bbbk[x])=Br(\Bbbk)=\{1\}$ by \cite[Theorem 7.5]{AG}. 
By \cite[Lemma 6.1]{Ne}, the $n$th Weyl algebra is a nontrivial 
Azumaya algebra over its center and that the center of the 
$n$th Weyl algebra is isomorphic to the polynomial ring of 
$2n$ variables. This implies that $Br(\Bbbk[x_1,\cdots,x_{2n}])$
is non-trivial for $n\geq 1$. Since $Br(A)$ is a subgroup of $Br(A[t])$, 
$Br(\Bbbk[x_1,\cdots,x_{m}])$ is non-trivial for all $m\geq 2$.
It is known that, by \cite{Ru}, $\Bbbk[x_1,x_2]$ is cancellative, 
while, by \cite{Gu1}, $\Bbbk[x_1,x_2,x_3]$ is not. 
\end{remark}

The study of connections between the cancellative property of $A$ 
and other invariants such as the Picard group or Grothendieck group
of $A$ might also be interesting.

\begin{remark}
\label{yyrem5.3} 
It would be interesting to know if retractable (respectively,
detectable, cancellative) property is preserved under some 
usual constructions. More precisely, suppose $A$ has property
${\mathcal P}$, where ${\mathcal P}$ stands for the {\it retractable}, 
{\it detectable}, or {\it cancellative} property.
\begin{enumerate}
\item[(a)]
Under what hypotheses, does a localization $AS^{-1}$ have 
${\mathcal P}$?
\item[(b)]
Under what hypotheses, does the skew polynomial extension 
$A[t,\sigma,\delta]$ have ${\mathcal P}$?
\item[(c)]
Under what hypotheses, does the matrix algebra $M_n(A)$ 
have ${\mathcal P}$?
\end{enumerate}
We are also wondering the following.
\begin{enumerate}
\item[(d)]
Suppose $A$ is filtered with associated graded ring $\gr A$. 
If $\gr A$ has property ${\mathcal P}$ (where 
${\mathcal P}$ stands for the {\it retractable}, 
{\it detectable}, or {\it cancellative} property), 
under what hypotheses, does $A$ have ${\mathcal P}$?
\item[(e)]
Let $J(A)$ be the Jacobson radical of $A$.
Suppose $A/J(A)$ has ${\mathcal P}$, under what hypotheses, 
does $A$ have ${\mathcal P}$?
\end{enumerate}
\end{remark}

The study of noncommutative Cancellation Problem has just started  
a couple of years ago. One can ask if the cancellative property
holds for any interesting class of noncommutative algebras. Here
are some examples.

\begin{question}
\label{yyque5.4}
Prove or disprove the following classes of algebras are 
cancellative.
\begin{enumerate}
\item[(1)]
Preprojective algebras.
\item[(2)]
The Weyl algebras in positive characteristic. Note that 
the Weyl algebras in characteristic zero are cancellative by
\cite[Proposition 1.3]{BZ1}.
\item[(3)]
The Sklyanin algebras of dimension $\geq 3$, or the 
PI Sklyanin algebras of dimension $\geq 3$.
\end{enumerate}
\end{question}

Several other questions have already been stated in the introduction.

\subsection*{Acknowledgments}
The authors thank Ken Goodearl and Daniel Krashen for many useful
conversations on the subject, and thank Ken Goodearl for suggesting
the term ``retractable'' in Definition \ref{yydef2.1} and Lance 
Small for providing the references \cite{BLK, Mal}. 
O. Lezama was partially
supported by Universidad Nacional de Colombia (HERMES code 40482).
Y.H. Wang was partially supported by the Foundation of China Scholarship 
Council (Grant No. [2016]3099), the Foundation of Shanghai Science and 
Technology Committee (Grant No. 15511107300), the Scientific Research 
Starting Foundation for the Returned Overseas Chinese Scholars of 
Ministry of Education of China and the Innovation program of Shanghai 
Municipal Education Commission (No. 15ZZ037) .
J.J. Zhang was partially supported by the US National Science Foundation
(Nos. DMS-1402863 and DMS-1700825).

\providecommand{\bysame}{\leavevmode\hbox to3em{\hrulefill}\thinspace}
\providecommand{\MR}{\relax\ifhmode\unskip\space\fi MR }
\providecommand{\MRhref}[2]{%

\href{http://www.ams.org/mathscinet-getitem?mr=#1}{#2} }
\providecommand{\href}[2]{#2}

\end{document}